\documentclass[12pt, reqno]{amsart}
\usepackage{amsmath, amsthm, amscd, amsfonts, latexsym, amssymb, graphicx, color}
\usepackage[bookmarksnumbered, colorlinks, plainpages]{hyperref}

\usepackage{amssymb}
\usepackage [all]{xy}
\usepackage{mathtools}
\usepackage{faktor}
\usepackage[percent]{overpic}
\usepackage{mathrsfs}

\DeclareMathOperator{\ima}{im}

\newcommand{\Free}{\mathscr{F}}
\newcommand{\Peano}{\mathscr{P}}

\newcommand{\I}{\mathcal{I}}
\newcommand{\Nat}{\mathbb{N}}

\newcommand{\length}[1]{\ensuremath{\lvert {#1} \rvert}}

\makeatletter \oddsidemargin.9375in \evensidemargin \oddsidemargin
\newtheorem{thm}{Theorem}[section]
\newtheorem{lem}[thm]{Lemma}
\newtheorem{prop}[thm]{Proposition}

\theoremstyle{definition}
\newtheorem{defn}[thm]{Definition}
\newtheorem{exa}[thm]{Example}

\theoremstyle{remark}

\numberwithin{equation}{section}

\begin{document}

\setcounter{page}{1}

%%%%%%%%%%%%%%%%%%%%%%%%%%%%%%%%%%%%%%%%%%%%%%%
%% Please do not remove the following statement.
%%%%%%%%%%%%%%%%%%%%%%%%%%%%%%%%%%%%%%%%%%%%%%%
%\noindent \textbf{{\footnotesize  Journal of Algebra and Related Topics \\
%Vol. XX, No XX, (201X), pp XX-XX}}\\[1.00in]
%%%%%%%%%%%%%%%%%%%%%%%%%%%%%%%%%%%%%%%%%%%%%%%%%%%%%%%%%%%%%%%%%%%%%%%
%%%%%%%%%%%%%%%%%%%%%%%%%%%%%%%%%%%%%%%%%%%%%%%%%%%%%%%%%%%%%%%%%%%%%
% Insert title of your article. Note: \title[short title]{full title}
%%%%%%%%%%%%%%%%%%%%%%%%%%%%%%%%%%%%%%%%%%%%%%%%%%%%%%%%%%%%%%%%%%%%%
\title[On semi-Peano algebras]{On semi-Peano algebras}
%%%%%%%%%%%%%%%%%%%%%%%%%%%%%%%%%%%%%%
% Author's name must be inserted here
%%%%%%%%%%%%%%%%%%%%%%%%%%%%%%%%%%%%%%
\author[Cardó]{C. Cardó}
%%%%%%%%%%%%%%%%%%%%%%%%
\thanks{{\scriptsize
\hskip -0.4 true cm MSC(2010): Primary: 08A60, 08B20; Secondary: 08B25.
%Received: 4 April 2010, Accepted: 25 June 2012.\\
$*$Corresponding author }}

%%%%%%%%%%%%%%%%%%%%%%%%%%%%%%%%%%%%%%%%%%%
\begin{abstract}
A semi-Peano algebra is an algebra for which each operation is injective, and the images of the operations are pairwise disjoint. The most straightforward non-trivial kind of finitely presented semi-Peano algebra are algebras with a single unary operation. There are two possible directions of generalization: algebras with a single operation of any arity, and algebras with several unary operations. The first can be solved easily by adapting results on equidecomposable groupoids from \cite{cardo2020equidecomposable}. However, the second way is somewhat different. We will show that a finitely presented multi-unary semi-Peano algebra is the free product of cyclic semi-Peano algebras and that a unique relation defines such cyclic algebras. In addition, we will characterize each cyclic algebra up to isomorphism.
\end{abstract}

%%%%%%%%%%%%%%%%%%%%%%%
\maketitle
%%%%%%%%%%%%%%%%%%%%%%%

%%%%%%%%%%%%%
\section {Introduction}\label{Notation}
%%%%%%%%%%%%%
Peano axioms of arithmetic of natural numbers can be understood as a cyclic algebra $\mathcal{N}=(N; s)$, such that the successor operation $s:N\longrightarrow N$ is injective, and such that the generator, that is, the zero, cannot be decomposed.
\emph{Generalized Peano algebras}, or simply here \emph{Peano algebras}, are a generalization of $\mathcal{N}$ for several operations and several generators:

\begin{defn} \label{Peano} An element of an algebra is said to be \emph{decomposable} if it belongs to the image of some operation.  A \emph{Peano algebra} is an algebra $\mathcal{A}=(A;F)$ satisfying the two following conditions:
\begin{enumerate}
\item[{\rm (i)}] for any $f,g \in F$, if $f(x_1, \ldots, x_n)=g(y_1,\ldots, y_m)$, then $f=g$, and $x_i=y_i$ for all $i=1, \ldots, n=m$.
\item[{\rm (ii)}] $\mathcal{A}$ is generated by its indecomposable elements. 
\end{enumerate}
When only {\rm (i)} is required, the algebra is said to be a \emph{semi-Peano algebra}. The class of those algebras is denoted by $\mathbf{SP}$. 
\end{defn}
The class of Peano algebras corresponds precisely to the class of absolutely free algebras; see \cite{diener1966order}, Proposition~3, for the statement, and \cite{Tandareanu2} for a proof. 
Similarly to natural numbers, Peano algebras provide an \emph{algebraic induction}. If a property $P$ is satisfied for every generator, and for each operation $f$ we have that:
\begin{align*}
x_1, \ldots, x_n \mbox{ satisfy } P \implies f(x_1, \ldots, x_n) \mbox{ satifies } P,
\end{align*}
then each element of the algebra satisfies the property $P$.
Peano algebras were introduced early to construct formal languages for algebraic logics with the possibility of infinitary connectives, see Diener \cite{diener1966order, diener1967note, diener1993predecessor, Dinier1995}. More recently, the format of free algebras as Peano algebras is used in computer science, in the area of knowledge representation through the concept of a labeled stratified graph, \cite{Tandareanu2000, Tandareanu2004, Tandareanu2}.

Let us fix some notation. $\mathcal{A}=(A; \{f_i\}_{i\in I})$ denotes an algebra with operations indexed by $I$.
When the index is not necessary, we will write $\mathcal{A}=(A;F)$, where $F=\{f_i\}_{i\in I}$.
From now on, a \emph{unary} algebra will mean a multi-unary algebra. A \emph{mono-unary algebra} is an algebra with only one operation. 
Absolutely free algebras over the generating set $X$ are denoted by $\Free(X)$. For cyclic free algebras with generator $x$ we write $\Free(x)$ or simply $\Free$.
$X \sqcup Y$ means $X\cup Y$, but with the condition that the sets are disjoint.
For other common concepts and notations on universal algebra, we follow \cite{sankappanavar1981course}.

Our primary interest is finitely presented semi-Peano algebras. Before starting, let us comment on some elementary properties and examples of such algebras. 
It is routine to prove that the class $\mathbf{SP}$ is closed under subalgebras and arbitrary direct products but not under quotients.
Finite non-nullary semi-Peano algebras have at most one operation. Consider two operations $f: A^n \longrightarrow A$ and $g: A^m \longrightarrow A$, such that $n\geq 1$ or $m\geq 1$. If $n\geq 1$, $f$ is injective. Since $A$ is finite, $f$ is bijective, and then $\ima f\cap \ima g \not= \emptyset$, which is absurd. If $m\geq 1$, the argument is similar. In particular, only the class of algebras with a single operation possesses trivial algebras. If an operation $f:A^n \longrightarrow A$ is injective and $A$ is finite, then $\length{A}^n\leq \length{A}$, which only is possible if $\length{A}=1$ or $0\leq n \leq1$.
Consequently, a non-trivial semi-Peano algebra is finite if and only if it is nullary or mono-unary.

Semi-Peano algebras are related to Jónsson-Tarski algebras, as shown in the following examples. 

\begin{exa}
A \emph{Jónsson-Tarski algebra} \cite{freese2020free, Jonsson1961Properties} is an algebra $\mathcal{J}=(J;+,\alpha,\beta)$, where $+$ is a binary operation and $\alpha, \beta$ are unary, satisfying the Cantor identities:
\begin{align*}
\alpha(x)+\beta(x)=x, \quad \alpha(x+y)=x, \quad \beta(x+y)=y.
\end{align*}
Since the operation $+:J^2 \longrightarrow J$ is necessarily bijective, reducts of Jónsson-Tarski algebras $(J;+)$ are semi-Peano groupoids. This inclusion of classes is proper. 
The free product of the trivial groupoid $\I$ by the cyclic free groupoid $\Free$ yields a semi-Peano groupoid which is not a free groupoid nor a reduct of a Jónsson-Tarski algebra: $\I * \Free \cong \langle a,b \mid a \approx a + a \rangle$.
\end{exa}

\begin{exa} \label{Unarisation} 
Jónsson-Tarski algebras can be generalized in the following obvious way, cf. \cite{dudek2012cantor}. Let $f$ be an $n$-ary operation and let  $p_1$, \ldots, $p_n$ be unary operations, satisfying the generalized Cantor identities:
$$p_i(f(x_1, \ldots, x_n))=x_i, \forall i=1, \ldots, n, \quad f(p_1(x), \ldots, p_n(x))=x.$$
Given an algebra with a single operation of arity $n>1$,  $\mathcal{A}=(A; f)$, by the \emph{unarisation of $\mathcal{A}$ on the $k$-th component}, denoted by $\mathfrak{U}_k(\mathcal{A})$, we mean the algebra: 
\begin{align*} \mathfrak{U}_k(\mathcal{A})=(A; \{f_{\vec{a}_k}\}_{ \vec{a}_k\in A^{n-1}}) \end{align*}
where $\vec{a}_k=(a_1, \ldots,a_{k-1}, a_{k+1},\ldots,  a_n)$ and 
\begin{align*}
f_{\vec{a}_k}(x)=f(a_1, \ldots, a_{k-1},x,a_{k+1}, \ldots, a_n).
\end{align*}
Then, the unarisation of the reduct on $f$ of a generalized Jónsson-Tarski algebra is a unary semi-Peano algebra, with possibly infinite operations. 
\end{exa}
The most straightforward non-trivial kind of finitely presented semi-Peano algebra are algebras with a single unary operation. The graph of those algebras consists of a finite set of directed cycles and ray graphs. An example is the algebra of natural numbers $\mathcal{N}$ whose graph is a ray. There are two direct generalizations of those algebras: (I) algebras with a single operation of any arity; and (II) algebras with several unary operations.

Regarding the way (I), the case of arity two was solved in \cite{cardo2020equidecomposable}. An \emph{equidecomposable groupoid} $\mathcal{M}=(M; +)$ is groupoid such that $x+y=x'+y'$ implies that $x=x', y=y'$, or in other terms, the operation $+$ is injective. Therefore, equidecomposable groupoids and semi-Peano groupoids are the same things. Finitely presented semi-Peano groupoids can be expressed in a normal form that characterizes them isomorphically. More specifically, if $X$ is a minimal generating set of a semi-Peano groupoid $\mathcal{M}$, there is a unique injective mapping $\Psi: X \longrightarrow \Free(X)$, such that for every $x\in X$, $x$ appears at least once in the term $\Psi(x)$, and such that $\mathcal{M}$ is the quotient of $\Free(X)$ by the congruence generated by the relations $(x,\Psi(x))$. In short, for each semi-Peano groupoid, there is a unique presentation without redundant relations; see \cite{cardo2020equidecomposable} for the details.
A generalization of that result for an operation of arity greater than two does not present a significant difficulty. It is only necessary to rewrite theorems from \cite{cardo2020equidecomposable} but with a more complicated notation. Thus, the way (I) does not offer much more mathematical interest.

Here we are concerned with the second case (II). We will show that any finitely presented unary semi-Peano algebra is the free product of cyclic unary semi-Peano algebras and that the congruences generated by a single relation define those cyclic algebras. This relation is given by the action of a word of the free monoid on the set of operations. Proving the decomposition is almost an elementary fact. So the contribution of this paper consists of the study of cyclic unary semi-Peano algebras.

%%%%%%%%%%%%%
\section {Some preliminary facts on semi-Peano algebras and unary algebras}\label{PreliminaryFacts}
%%%%%%%%%%%%%

We will need a characterization of semi-Peano algebras in terms of congruences. 

\begin{defn} A congruence $\theta$ of an algebra $\mathcal{U}=(U;F)$ is said to be:
\begin{enumerate}
\item[{\rm (i)}] \emph{closed} if for each $f \in F$:
\begin{align*} 
(f(x_1, \ldots, x_n), f(y_1, \ldots, y_n)) \in \theta \implies (x_1, y_1), \ldots, (x_n,y_n) \in \theta.
\end{align*}
\item[{\rm (ii)}]  \emph{balanced} if for each $f,g \in F$:
\begin{align*}
(f(x_1, \ldots, x_n), g(y_1, \ldots, y_m)) \in \theta \implies f=g.
\end{align*}
\end{enumerate}
\end{defn}

\begin{lem} \label{LemaEquiClosed} Given an homomorphism of algebras $\Phi$, $\ima \Phi$ is a semi-Peano algebra if and only if $\ker \Phi$ is a closed and balanced congruence.
\end{lem}
\begin{proof} Let $\Phi: \mathcal{U} \longrightarrow \mathcal{V} $ be a homomorphism of algebras and $\mathcal{U}=(U;F)$. We have the following equivalences, where each statement must be read with the prefix ``for all $x_1, \ldots, x_n, y_1, \ldots, y_m \in U$'', and $n$ and $m$ are the corresponding arities: \begin{align*}
& \ima \Phi \mbox{ is semi-Peano}\\
\iff & \mbox{ if } f(\Phi(x_1), \ldots, \Phi(x_n))=g((\Phi(y_1), \ldots, \Phi(y_m)) \mbox{ with } f,g \in F, \\
&\mbox{ then } f=g, n=m \mbox{ and } \Phi(x_1)=\Phi(y_1), \ldots, \Phi(x_n)=\Phi(y_n)\\
\iff & \mbox{ if } \Phi( f(x_1, \ldots, x_n))=\Phi(g(y_1, \ldots, y_m)) \mbox{ with } f,g \in F, \\
&\mbox{ then } f=g, n=m \mbox{ and } \Phi(x_1)=\Phi(y_1), \ldots, \Phi(x_n)=\Phi(y_n)\\
\iff & \mbox{ if } (f(x_1, \ldots, x_n), g(y_1, \ldots, y_m)) \in \ker \Phi \mbox{ with } f,g \in F, \\
&\mbox{ then } f=g, n=m \mbox{ and } (x_1, y_1), \ldots, (x_n, y_n) \in \ker \Phi\\
\iff & \ker \Phi \mbox{ is closed and balanced}. %\qedhere
\end{align*}
\end{proof}

$\mathbf{SP}$ possesses $\mathbf{SP}$-free algebras. Every absolutely free algebra is a Peano algebra, which, in its turn, is a semi-Peano algebra. Since absolutely free algebras satisfy the universal mapping property, the class of $\mathbf{SP}$-free algebras is the class of absolutely free algebras.
Thus, each semi-Peano algebra $\mathcal{U}$ is of the form $\Free(X)/\theta$, where $\theta$ is a closed and balanced congruence, and $X$ is a generator set of $\mathcal{U}$. The arbitrary intersection of closed and balanced congruences is closed and balanced. Therefore, the following definition makes sense.

\begin{defn} Given an absolutely free algebra $\Free(X)$ and a subset $R\subseteq \Free(X)^2$, the least closed and balanced congruence on $\Free(X)$ containing $R$ is denoted by $\boxminus(R)$. Congruences generated by only one relation $(a,b)$ are abbreviated $\boxminus(a,b)=\boxminus(\{(a,b)\})$. A semi-Peano algebra $\mathcal{U}$ is \emph{finitely presented} if there are finite sets $X$ and $R$ such that $\mathcal{U}\cong \Free(X)/\boxminus(R)$.
\end{defn}

In addition, we  will need some elementary facts on unary algebras. Given a unary algebra $\mathcal{U}=(U; F)$, a term function is a mapping of the form $\alpha_1\circ \cdots \circ \alpha_n: U \longrightarrow U$, for some operations  $\alpha_1, \ldots, \alpha_n \in F$. Consider the free monoid $F^*$ of words over the alphabet $F$. There is a natural action of the monoid $F^*$ over the set $U$ defined by $
\alpha_1 \cdots \alpha_n(x)=(\alpha_1 \circ \cdots \circ \alpha_n) (x)$, and if $1$ denotes the empty word, then $1(x)=x$.  
The length of a word $\varphi \in F^*$ is denoted by $\length{ \varphi }$. 
Given two words $\varphi$ and $\psi$ such that $\varphi$ is a prefix of $\psi$,  $\varphi {\setminus} \psi$ denotes the unique word such that $\varphi (\varphi {\setminus} \psi)=\psi$. When $\psi$ is a suffix of $\varphi$,  $\varphi/ \psi$ denotes the unique word such that $(\varphi / \psi)\psi=\varphi$.
Two words $\omega$ and $\omega'$ are \emph{conjugated} if there is a word $\varphi$ such that $\omega \varphi= \varphi\omega'$.
This is equivalent to saying that there are words $\alpha, \beta$ such that $\omega=\alpha\beta$, $\omega'=\beta\alpha$.

The \emph{free product} of two unary algebras $\mathcal{U}=(U; \{f_i\}_{i \in I})$ and $\mathcal{V}=(V; \{g_i\}_{i \in I})$ is  defined by $\mathcal{U}\amalg \mathcal{V}=(U\sqcup V; \{f_i \sqcup g_i\}_{i \in I})$ where:
\begin{align*}
f_i \sqcup g_i (x)= \begin{cases} f_i(x) & \mbox{if } x \in U, \\ g_i(x) & \mbox{if } x \in V. \end{cases}
\end{align*}
The product can be extended for an arbitrary family of algebras $\coprod_{j\in J} \mathcal{U}_j$. An algebra $\mathcal{W}$ is said to be $\amalg $-irreducible when $\mathcal{W}=\mathcal{U}\amalg \mathcal{V}$ implies that $\mathcal{U}=\emptyset$ or $\mathcal{V}=\emptyset$.
Consider the following relation in a unary algebra: $x \sim y$ if and only if there is a word $\varphi$ such that $\varphi(x)=y$ or $\varphi(y)=x$. $\sim$ is an equivalence relation and defines a partitition of the algebra into $\amalg$-irreducible subalgebras. Cyclic unary algebras are always $\amalg$-irreducible. The reciprocal is not in general true, except for semi-Peano algebras, as it is shown in the Proposition~\ref{Proposition1}.

%%%%%%%%%%%%%%%%%%%%%%%%%%%%%
\section{Classification of unary semi-Peano algebras}
%%%%%%%%%%%%%%%%%%%%%%%%%%%%%

\begin{lem} \label{LemaPrevi} If $\varphi(x)=\psi(y)$ in a unary semi-Peano algebra and $\length{\varphi}\leq \length{\psi}$, then $\varphi$ is a prefix of $\psi$ and $x=\varphi {\setminus} \psi (y)$.
\end{lem}
\begin{proof} Apply Definition~\ref{Peano}(i) $n$ times, where $n=\length{\varphi}$.
%Let $\varphi=\alpha_1 \cdots \alpha_n$, $\psi=\beta_1\cdots \beta_m$, $n\leq m$, with $\alpha_i, \beta_j \in F$. By Definition~\ref{Peano} of semi-Peano algebras:
%\begin{align*}
%\varphi(x)=\psi(y)  &\implies \alpha_1 \cdots \alpha_n(x)=\beta_1\cdots \beta_m(y)\\
%& \implies \alpha_1=\beta_1 \mbox{ and } \alpha_2 \cdots \alpha_n(x)=\beta_2\cdots \beta_m(y)\\
%& \implies \alpha_2=\beta_2 \mbox{ and } \alpha_3 \cdots \alpha_n(x)=\beta_3\cdots \beta_m(y)\\
%& \quad \quad \,\,\,\, \mbox{(repeating $n$ times) }\\
%& \implies \alpha_n=\beta_n \mbox{ and } x=\beta_{n+1}\cdots \beta_m(y)=\varphi {\setminus} \psi (y). 
%\end{align*}
\end{proof}

\begin{prop} \label{Proposition1} Let $\mathcal{U}=(U;\{f_i\}_{i\in I})$ be a unary semi-Peano algebra with a minimal generating set $A$. Then:
\begin{align*}
\mathcal{U}=\coprod_{a \in A} \mathcal{U}_a,
\end{align*}
where $\mathcal{U}_a$ is the cyclic subalgebra $\mathcal{U}_a=(\langle a \rangle; \{f_{a,i}\}_{i \in I})$ and $f_{a,i}$ is the operation $f_i$ restricted to $\langle a \rangle$.
\end{prop}
\begin{proof} First we see that $\langle A \rangle =\bigcup_{a \in A}  \langle a \rangle$. The direction $\supseteq$ is trivial. For the direction $\subseteq$, we notice that any combination of the operations can only be applied over only one element, since they are unary. Thus, for every $x\in \langle A \rangle$ there is a word $\varphi$ and an element $a \in A$, such that $\varphi(a)=x$. 

Next we check that the unions are disjoint. Let $a, b \in A$ be two different generators and supose that $\langle a \rangle \cap \langle b \rangle \not=\emptyset$. This means that there are two words, $\varphi$ and $\psi$, such that $\varphi(a)=\psi(b)$. Without lost of generality, we assume that $\length{\varphi} \leq \length{\psi}$. By Lemma~\ref{LemaPrevi}, $a=\varphi {\setminus} \psi (b)$, which is a contradiction with the minimality of the generating set $A$, because the generator $a$ can be removed. It is routine to prove that  each operation $f_i$ of $\mathcal{U}$ can be written as $f_i=\bigsqcup_{a \in A} f_{a, i}$. 
\end{proof}

By Proposition~\ref{Proposition1}, a unary semi-Peano algebra is $\amalg$-irreducible if and only if it is cyclic. In addition, all the minimal generating sets are equinumerous, whereby it makes sense to define the \emph{rank} of a  unary semi-Peano algebra as the cardinal of a minimal generating set. 

We will abbreviate $\Peano/\omega=\Free (a_0)/\boxminus( a_0, \omega(a_0))$, where $\omega \in F^*$ and $F$ is the set of operations of $\Free(a_0)$. We call the \emph{canonical generator} of $\Peano/\omega$ the equivalence class containing $a_0$.

\begin{lem}\label{Lemma2} Any cyclic and finitely presented unary semi-Peano algebra is of the form $\Peano/\omega$ for some word $\omega$.
\end{lem}
\begin{proof} A cyclic and finitely presented unary semi-Peano algebra $\mathcal{U}$  is of the form:
\begin{align*}
\mathcal{U}\cong \faktor{\Free(c) }{\boxminus(\{(a_i,b_i) \mid i \in I\})},
\end{align*}
for some finite set $I$. Let $a \approx b$ denote two elements $a,b \in \Free(c)$ such that evaluates the same in $\mathcal{U}$. Consider the congruence generated by a single relation $\boxminus(a,b)$. Since $\mathcal{U}$ is cyclic, there are words $\varphi, \psi$ such that $\varphi(c)\approx a$ and $\psi(c)\approx b$. By Lemma~\ref{LemaPrevi}, $c \approx \varphi {\setminus} \psi (c)$ or $c \approx \psi {\setminus} \varphi (c)$. That is, $c \approx \phi(c)$ for some word $\phi$. Thus:
\begin{align*}
\boxminus(a,b)=\boxminus(\varphi(c), \psi(c))=\boxminus(c,\phi(c)).
\end{align*}
This means that $\mathcal{U}$ we can be rewritten as:
\begin{align*}
\mathcal{U}\cong\faktor{\Free(c) }{  \boxminus( \{(c,\phi_i(c)) \mid i \in J\} )},
\end{align*}
where $J=\{ i \in I \mid a_i \not \approx b_i\}$. In the case that $J=\emptyset$, $\mathcal{U}=\Free(c)=\Peano/1$ and we are done. 
So, we assume that $\length{J}>1$. Given a set of relations $R \subseteq \{ (c, \varphi(c)) \mid \varphi\not=1\}$, we introduce the \emph{weight} of $R$ as:
\begin{align*}
W (R)=\sum_{(c, \varphi(c)) \in R} \length{\varphi}. 
\end{align*}
Suppose that there are two relations $(c,\varphi(c)), (c,\psi(c)) \in R$ such that $\varphi$ and $\psi$ have different length, say, $\length{\varphi}<\length{\psi}$. Then we define:
\begin{align*}
R'=R \setminus \{(c, \psi(c))\} \cup \{(c, \varphi{\setminus}\psi (c))\}.
\end{align*}
Notice that $\boxminus(R)=\boxminus(R')$, but 
\begin{align*}
W(R')=W(R)- \length{\psi} + (\length{\psi}-\length{\varphi})=W(R)-\length{\varphi}<W(R),
\end{align*}
because $\varphi\not=1$. We denote by a superscript $k)$ the successive application of that transformation: $R, R^{1)}, R^{2)}, R^{3)}, \ldots$. Since $R=\{(c, \phi_i(c)) \mid i \in J\}$ is finite, $W(R)$ is finite, which means that we can only apply the transformation finite many times, say $k\geq0$. In that point, $\boxminus(R^{k)})=\boxminus(R)$, and for each $(c,\varphi(c)), (c, \psi(c)) \in R^{k)}$, $\length{\varphi}=\length{\psi}$. But then, by Lemma~\ref{LemaPrevi}, $\varphi(c) \approx \psi(c)$ implies that $\varphi$ is a prefix of $\psi$ and vice versa. Therefore, $\varphi=\psi$, and $R^{k)}$ only contains one relation, say $(c, \omega(c))$, and:
\begin{align*}
\mathcal{U}\cong\faktor{\Free(c) }{\boxminus(c,\omega(c))}=\faktor{\Peano}{\omega},
\end{align*}
for some word $\omega$. 
\end{proof}

\begin{lem} \label{LemaLlarg} Let $\Peano/\omega$ be a cyclic semi-Peano algebra, and $a$ be the canonical generator.
\begin{enumerate}
\item[{\rm (i)}] If $\varphi(a)=a$ for some word $\varphi$, then $\varphi=\omega^n$ for  some $n\geq 0$. 
\item[{\rm (ii)}]  \label{LemaLlarg1} If $\varphi(a)=\psi(a)$ for some words $\varphi, \psi$, with $\length{\varphi}\leq \length{\psi}$, then $\psi=\varphi \omega^n$ for  some $n\geq 0$. 
\item[{\rm (iii)}] \label{LemaLlarg2} If $\varphi(b)=b$ for some $b \in \Peano/\omega$ and some word $\varphi$, then $\varphi$ and $\omega^n$ are conjugated for some $n\geq 0$. 
\end{enumerate}
\end{lem}

\begin{proof} (i) Let $\Peano/\omega= \Free(a_0)/\boxminus( a_0, \omega(a_0))$. The equivalence class $a=\{a_0, \omega(a_0), \omega^2(a_0), $ $\omega^3(a_0), \ldots \}$ is the canonical generator of $\Peano/\omega$. Suppose that $\varphi(a)=a$, for some word $\varphi$. Then, $\varphi(a_0) \in a$, which means that $\varphi(a_0)=\omega^n(a_0)$. Since  $\Free(a_0)$ is absolutely free, $\varphi=\omega^n$, for some $n\geq0$. 

(ii) If $\varphi(a)=\psi(a)$, then by Lemma~\ref{LemaPrevi}, $\varphi{\setminus} \psi(a)=a$. By (i), $\varphi{\setminus} \psi=\omega^n$, that is, $\psi=\varphi \omega^n$. 

(iii) Since $a$ is a generator, $b=\alpha(a)$, for some word $\alpha$. Then, $\alpha(a)=\varphi(\alpha(a))$, which means that $a=\alpha{\setminus} \varphi \alpha(a)$. By (i), $\alpha {\setminus} \varphi \alpha=\omega^n$ for some $n\geq 0$, whereby $\varphi\alpha=\alpha\omega^n$. 
\end{proof}

\begin{prop} \label{Proposition2} $\Peano/\omega\cong \Peano/\omega'$ if and only if  $\omega$ and $\omega'$ are conjugated words.
\end{prop}
\begin{proof} $(\Rightarrow)$ Let $\Phi: \Peano/\omega \longrightarrow \Peano/\omega'$ be an isomorphism, $a$, the generator of $\Peano/\omega$, and $b=\Phi(a)$. Since $\omega(a)=a$, applying the homomorphism we get $\Phi(\omega(a))=\omega(\Phi(a))=\Phi(a)$, that is, we have the identity $\omega(b)=b$ in  $\Peano/\omega'$. By Lemma~\ref{LemaLlarg}(iii), $\omega$ and $\omega'^n$ are conjugated for some $n\geq 0$. Now, by considering the inverse isomorphism $\Phi^{-1}: \Peano/\omega' \longrightarrow \Peano/\omega$, and applying the same reasoning, $\omega'$ and $\omega^m$ are conjugated for some $m\geq 0$. 

If $\omega=\omega'=1$, we are done. So, let us suppose that at least one of the words, say $\omega'$, is different to $1$. 
Join both conjugations: $\omega \alpha= \alpha \omega'^n$ and $\omega' \beta= \beta \omega^m$. 
From the first equality, $\omega= \alpha\omega'^n /\alpha$. Substituying in the second equality and taking lengths: 
\begin{align*}
& \omega' \beta=\beta (\alpha\omega'^n /\alpha)^m \\
\implies &\length{\omega'}+\length{\beta}=\length{\beta}+m(\length{\alpha} + n\length{\omega'}-\length{\alpha})\\
 \implies & \length{\omega'}=mn\length{\omega'} \implies mn=1 \implies m=n=1,
\end{align*}
where we have used that $\omega'\not=1$, that is, $\length{\omega'}\not=0$. Therefore, $\omega$ and $\omega'$ are conjugated.  

$(\Leftarrow)$ Suppose that $\omega=\beta\alpha$ and $\omega'=\alpha\beta$.  Consider the mapping $\Phi: \Peano/\omega \longrightarrow \Peano/\omega'$ defined by $\Phi(x)=\varphi \beta (b)$, where $\varphi(a)=x$, and where $a$ and $b$ are the canonical generators of $  \Peano/\omega$ and $\Peano/\omega'$, respectively. First we check that the mapping is well-defined. Suppose that $x=\varphi(a)=\psi(a)$, with $\length{\varphi}\leq \length{\psi}$. By Lemma~\ref{LemaLlarg}~(ii), $\psi=\varphi \omega^n$, for some $n\geq0$. Therefore, 
\begin{align*}
\Phi(\psi(a))&=\Phi(\varphi\omega^n(a))=\varphi\omega^n \beta(b)=\varphi(\beta\alpha)^n \beta(b)=\varphi \beta (\alpha\beta)^n (b)=\varphi\beta(b) \\
&=\Phi(\varphi(a)).
\end{align*}
Clearly $\Phi$ is an homomorphism. If $x=\varphi(a)$, then  for each $f\in F$,  $\Phi(f(x))=\Phi(f(\varphi(a)))=\Phi(f\varphi(a))=f\varphi\beta(b)=f (\varphi\beta(b))= f(\Phi(a))=f(\Phi(x))$.
Next, we see that $\Phi$ is an epimorphism. Since $\alpha (\beta(b)) \in \langle \beta(b)\rangle$, we have:
\begin{align*}
\langle \alpha(\beta(b))\rangle \subseteq \langle \beta(b) \rangle \subseteq \langle b \rangle
& \implies \langle \alpha\beta(b)\rangle =\langle b \rangle \subseteq \langle \beta(b) \rangle \subseteq \langle b \rangle \\
& \implies  \langle b \rangle= \langle \beta(b) \rangle = \langle \Phi(a) \rangle. \end{align*}
Finally we check that $\Phi$ is a monomorphism. Suppose $x=\varphi(a)$ and $y=\psi(a)$ with $\length{\varphi} \leq \length{\psi}$. Then:
\begin{align*}
 \Phi(x)=\Phi(y) &\implies \Phi(\varphi(a))=\Phi(\psi(a)) \implies \varphi\beta (a)=\psi \beta (a) \\
&\implies  (\varphi \beta){\setminus} (\psi \beta)=\omega'^n \\
&\implies \psi \beta=\varphi \beta \omega'^n =\varphi\beta(\alpha\beta)^n=\varphi (\beta\alpha)^n \beta,
\end{align*}
for some $n\geq 0$. Since we are in the free monoid $F^*$, we can cancel $\beta$'s:
\begin{align*}
\psi\beta=\varphi (\beta\alpha)^n \beta \implies \psi=\varphi(\beta\alpha)^n,
\end{align*}
and then $\psi(a)=\varphi(\beta\alpha)^n(a)=\varphi( \omega^n(a))=\varphi(a)$. That is, $y=x$. 
\end{proof}

Now we can classify all the finitely presented unary semi-Peano algebras. 

\begin{thm} \label{FirstDecomp} For any finitely presented unary semi-Peano algebra $\mathcal{U}$ with rank $n$ there are $n$ words $\omega_1, \ldots, \omega_n$ such that:
\begin{align*}
\mathcal{U} \cong \faktor{\Peano}{\omega_1} \amalg  \cdots \amalg  \faktor{\Peano}{\omega_n},
\end{align*}
and the decomposition is unique, up to the order of the factors and conjugation of the words.  
\end{thm}
\begin{proof} Let $A=\{a_1, \ldots, a_n\}$ be a minimal generating set of $\mathcal{U}$. By Proposition~\ref{Proposition1}, $\mathcal{U} = \mathcal{U}_{a_1} \amalg  \cdots \amalg  \mathcal{U}_{a_n}$, 
where $\mathcal{U}_{a_1}, \ldots, \mathcal{U}_{a_n}$ are cyclic. By Lemma~\ref{Lemma2}, each $\mathcal{U}_{a_i}$ is isomorphic to $\Peano/\omega_i$ for some word $\omega_i$. The uniqueness of the decomposition follows from Propositions~\ref{Proposition1} and~\ref{Proposition2}.
\end{proof}

\begin{exa} Consider the bijection $\diamond: \Nat \times \Nat \longrightarrow \Nat$:
\begin{align*}
x \diamond y= \frac{1}{2} ( x^2+y^2+2xy -x-3y+2 ).
\end{align*}
The groupoid $\mathcal{G}=(\Nat; \diamond)$ is the reduct of a Jónsson-Tarski algebra, and it has only two non-trivial relations:
\begin{align*}
1\diamond 1=1 \quad \mbox{ and } \quad 1 \diamond 2=2,
\end{align*}
The unarisation on the second component $\mathfrak{U}_2(\mathcal{G})$ gives a unary semi-Peano algebra, recall Example~\ref{Unarisation}, with infinite countable unary operations $\diamond_a(x)=a\diamond x$, for each $a \in \Nat$.
The above relations in $\mathcal{G}$ become the unary relations in $\mathfrak{U}_2(\mathcal{G})$:
\begin{align*}
\diamond_1(1)=1 \quad \mbox{ and } \quad \diamond_1(2)=2.
\end{align*}
It can be proved that $1$ and $2$ are the generators of $\mathfrak{U}_2(\mathcal{G})$:
\begin{align*}
\Nat = \langle 1 \rangle \sqcup \langle 2 \rangle =\{1,3,4,6,8,10,13,16, \ldots\} \sqcup \{2,5,9,11,14,17,\ldots\},
\end{align*}
and therefore:
\begin{align*}
\mathfrak{U}_2(\mathcal{G}) \cong \faktor{\Peano}{\diamond_1} \amalg \faktor{\Peano}{\diamond_1}.
\end{align*}
If we take the unarisation on the first component, the operations are of the form $\widetilde{\diamond}_a=x \diamond a$. Only one generator is needed, $1$, and then:
\begin{align*}
\mathfrak{U}_1(\mathcal{G}) \cong \faktor{\Peano}{\widetilde{\diamond}_1}. 
\end{align*}
\end{exa}

%%%%%%%%%%%%%%%%%%%%%%
\begin{figure}[tb] %choose h,b,t (here, bottom, top)
\centering
\begin{overpic}[scale=0.15]{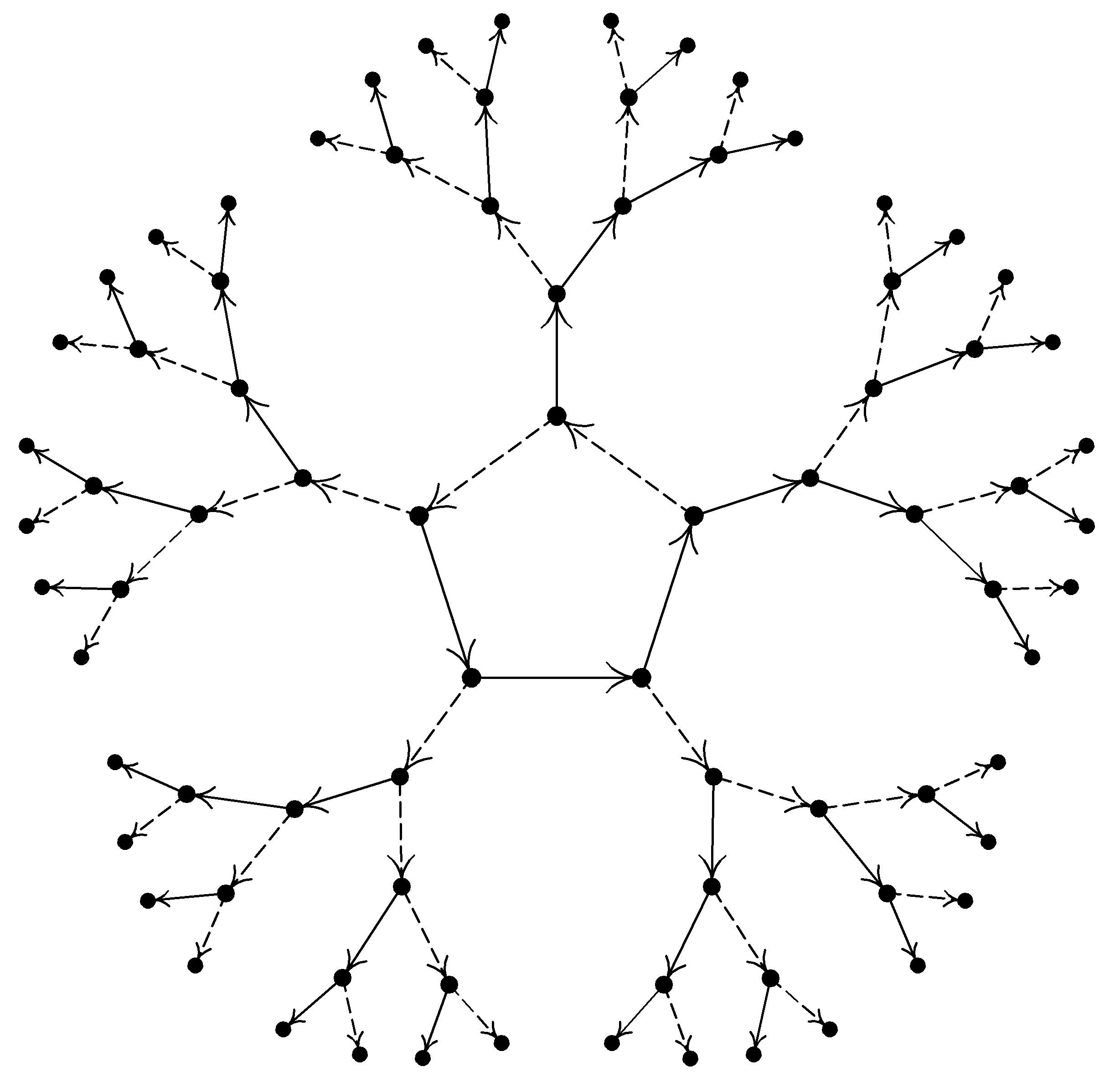}
 \put (35,43) {$f$} 
 \put (61,42) {$f$}
 \put (56,57) {$g$}
 \put (42,57) {$g$}
 \put (48,31) {$f$}
\end{overpic}
\caption{A portion of the graph of $\Peano/fg^2f^2$. }\label{Pentagon}
\end{figure}
%%%%%%%%%%%%%%%%%%%%%%

%%%%%%%%%%%%%%%%%%%%%%
\begin{figure}[tb] %choose h,b,t (here, bottom, top)
\centering
\begin{overpic}[scale=0.15]{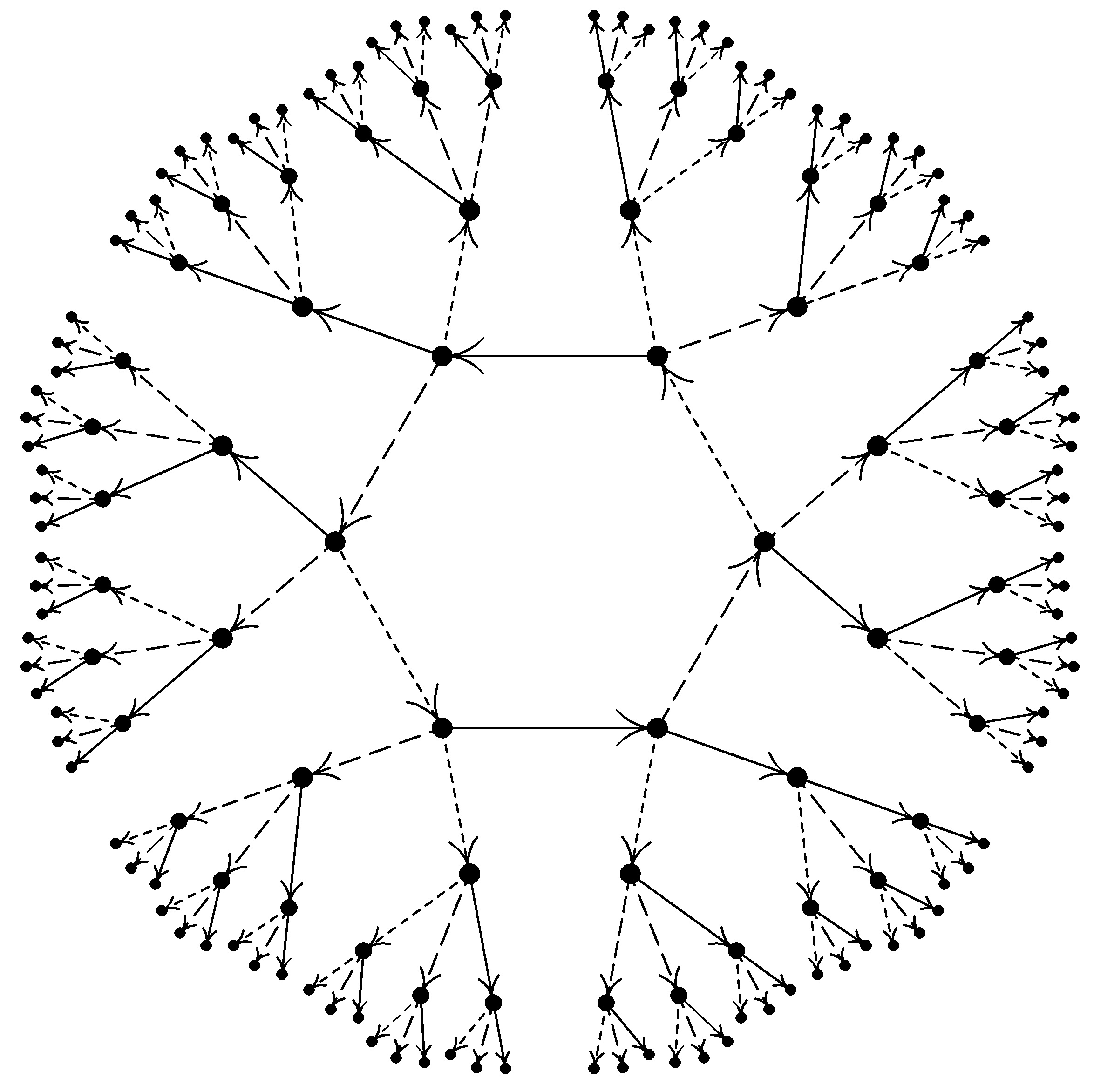}
  \put (50,35) {$f$} 
  \put (63,43) {$g$} 
  \put (62,55) {$h$} 
  \put (50,63) {$f$} 
  \put (36,55) {$g$} 
  \put (36,43) {$h$} 
\end{overpic}
\caption{A portion of the graph of $\Peano/(fgh)^2$.}\label{Hexagon}
\end{figure}
%%%%%%%%%%%%%%%%%%%%%%

%%%%%%%%%%%%%%%%%%%%%%
\begin{figure}[tb] %choose h,b,t (here, bottom, top)
\centering
\begin{overpic}[scale=0.11]{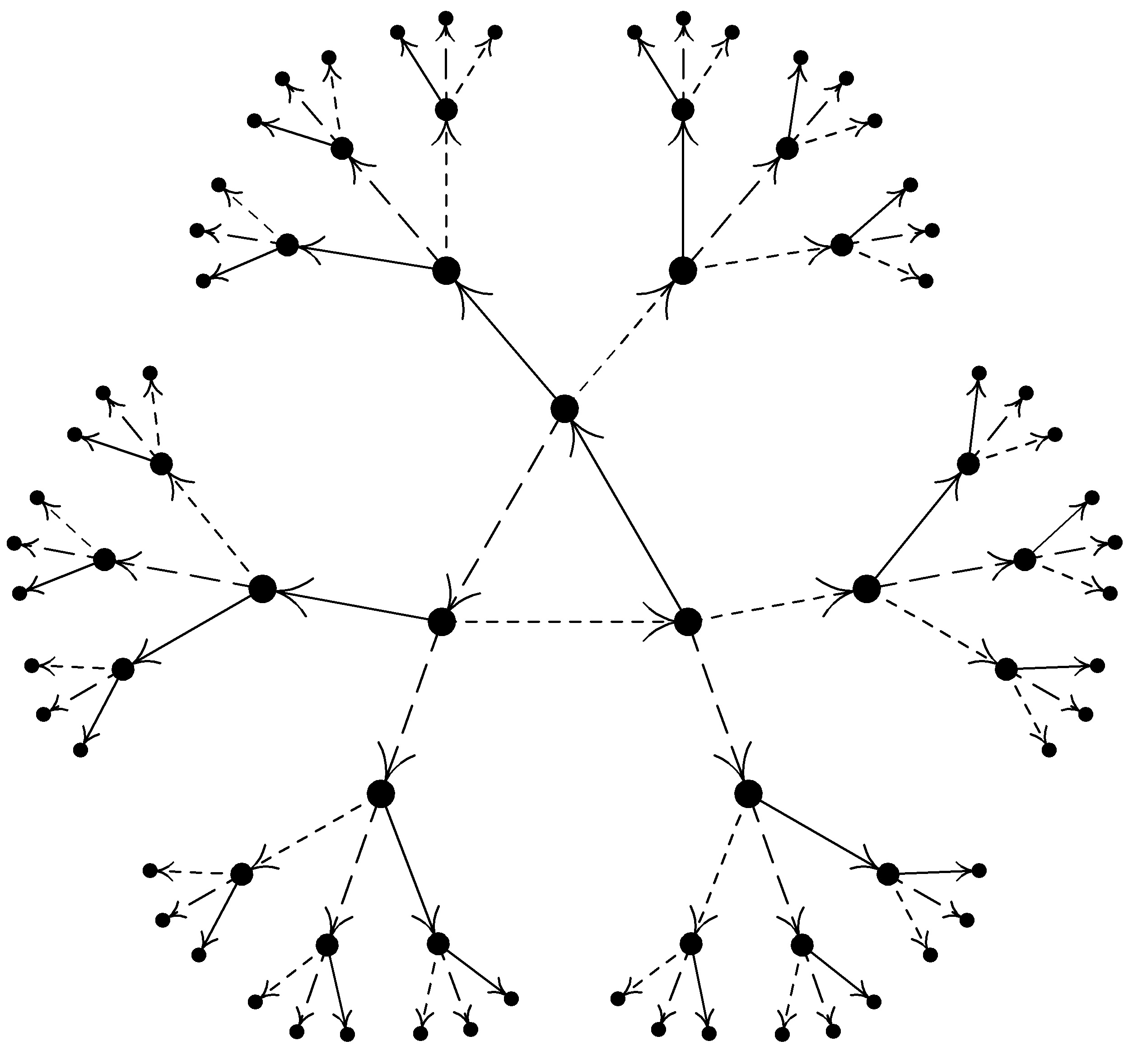}

  \put (59,47) {$f$} 
  \put (39,47) {$g$}
   \put (48,33) {$h$} 
\end{overpic}
\caption{A portion of the graph of $\Peano/fgh$.}\label{Triangle}
\end{figure}
%%%%%%%%%%%%%%%%%%%%%%

The graph of a unary algebra gives us nice representations of semi-Peano algebras.
If $\omega\not=1$, the graph of $\Peano/\omega$ contains a unique cycle defined by the word $\omega$, see Figure~\ref{Pentagon}. Every vertex of that cycle is an alternative generator of $\Peano/\omega$. The cycle determines the rest of the graph, which can be completed by adding trees. We can check many properties visually. For instance, all the proper subalgebras of a cyclic unary semi-Peano algebra are absolutely free. Or, for instance, If $q$ divides $p$, there is an epimorphism $\Peano/\omega^p \longrightarrow \Peano/\omega^q$. Thus, $\Peano/fgh$ is a quotient of $\Peano/(fgh)^2$; see Figure~\ref{Hexagon} and \ref{Triangle}. The algebra $\Peano/fg^2f^2$ has no proper quotients, Figure~\ref{Pentagon}.

As we have seen, when we consider semi-Peano algebras with several operations (II), the presentation is no longer unique, as in the case of algebras with one operation (I). That suggests that the general case could be much more complex than a simple combination of the ideas of (I) and (II).

%------------------------------------------------------------------------------------%

%\vskip 0.4 true cm

%\begin{center}{\textbf{Acknowledgments}}
%\end{center}
%The authors would like to thank the referee for careful reading. \\ \\
%\vskip 0.4 true cm

%------------------------------------------------------------------------------------%
%---------------------------------------------------------------------------------------%
%%%%%%%%%%%%%%%%%%%%%%%%%%%%%%%%%%%%%%%%%%%
% References
%%%%%%%%%%%%%%%%%%%%%%%%%%%%%%%%%%%%%%%%%%%
\bibliographystyle{amsplain}
%%%%%%%%%%%%%%%%%%%%%%%%%%%%%%%%%%%%%%%%%%%
% Please cite your relevant papers but at most total 5 papers/books.
%%%%%%%%%%%%%%%%%%%%%%%%%%%%%%%%%%%%%%%%%%%

\bigskip
\bigskip

{\footnotesize {\bf First Author}\; \\ {Departament de Ci\`encies Bàsiques}, {Universitat Internacional de Catalunya, c/ Josep Trueta s/n,} {Sant Cugat del Vallès,  08195, Spain.}\\
{\tt Email: ccardo@uic.es}\\

\end{document}